\documentclass[a4paper,12pt]{amsart}
\usepackage{amsfonts}
\usepackage{amssymb}
\usepackage{ifthen}
\usepackage{graphicx}
\nonstopmode \numberwithin{equation}{section}
\usepackage[usenames]{color}

\newtheorem{theorem}[equation]{Theorem}
\newtheorem{lemma}[equation]{Lemma}
\newtheorem{corollary}[equation]{Corollary}

\numberwithin{equation}{section}
\newtheorem{case}[equation]{Case}

\newtheorem{claim}[equation]{Claim}
\theoremstyle{definition}
\newtheorem{definition}[equation]{Definition}
\newtheorem{remark}[equation]{Remark}

\newtheorem{examples}[equation]{Examples}

\newtheorem{thm}{Theorem}[section]
\newtheorem{lem}{Lemma}[section]
\newtheorem{cor}{Corollary}[section]

\newtheorem{cl}{Claim}[section]
\newtheorem{ca}{Case}[section]
\newtheorem{sca}{Subcase}[section]
\newtheorem{scl}[section]{Subclaim}
\newtheorem{conj}[equation]{Conjecture}

\theoremstyle{definition}
\newtheorem{defn}{Definition}[section]
\newtheorem{op}[equation]{Open Problem}
\newtheorem{ques}[equation]{Question}
\newtheorem{exam}[equation]{Example}

\newcounter {own}
\def\theown {\thesection       .\arabic{own}}

\newenvironment{pf}[1][]{%
 \vskip 3mm
 \noindent
 \ifthenelse{\equal{#1}{}}%
  {{\slshape Proof. }}%
  {{\slshape #1.} }%
 }%
{\qed\bigskip}

\newcounter{alphabet}
\newcounter{tmp}

\makeatletter
\newcommand{\Ref}[1]{\@ifundefined{r@#1}{}{\setcounter{tmp}{\ref{#1}}\Alph{tmp}}}
\makeatother
\newcommand{\IR}{{\mathbb R}}

\newcommand{\diam}{{\operatorname{diam}}}

\newcommand{\dist}{{\operatorname{dist}}}

\def\be{\begin{equation}}
\def\ee{\end{equation}}

\newcommand{\bee}{\begin{enumerate}}
\newcommand{\eee}{\end{enumerate}}

\newcommand{\blem}{\begin{lem}}
\newcommand{\elem}{\end{lem}}
\newcommand{\bthm}{\begin{thm}}
\newcommand{\ethm}{\end{thm}}
\newcommand{\bcor}{\begin{cor}}
\newcommand{\ecor}{\end{cor}}
\newcommand{\beg}{\begin{exam}}
\newcommand{\eeg}{\end{exam}}
\newcommand{\begs}{\begin{examples}}
\newcommand{\eegs}{\end{examples}}
\newcommand{\bdefe}{\begin{defn}}
\newcommand{\edefe}{\end{defn}}
\newcommand{\bprob}{\begin{prob}}
\newcommand{\eprob}{\end{prob}}
\newcommand{\bques}{\begin{ques}}
\newcommand{\eques}{\end{ques}}
\newcommand{\bei}{\begin{itemize}}
\newcommand{\eei}{\end{itemize}}
\newcommand{\bcon}{\begin{conj}}
\newcommand{\econ}{\end{conj}}
\newcommand{\bop}{\begin{op}}
\newcommand{\eop}{\end{op}}

\newcommand{\bca}{\begin{ca}}
\newcommand{\eca}{\end{ca}}
\newcommand{\bsca}{\begin{sca}}
\newcommand{\esca}{\end{sca}}

\newcommand{\bcl}{\begin{cl}}
\newcommand{\ecl}{\end{cl}}

\newcommand{\bscl}{\begin{scl}}
\newcommand{\escl}{\end{scl}}

\newcommand{\bcons}{\begin{conjs}}
\newcommand{\econs}{\end{conjs}}
\newcommand{\bprop}{\begin{propo}}
\newcommand{\eprop}{\end{propo}}
\newcommand{\er}{\end{rem}}
\newcommand{\brs}{\begin{rems}}
\newcommand{\ers}{\end{rems}}
\newcommand{\bo}{\begin{obser}}
\newcommand{\eo}{\end{obser}}
\newcommand{\bos}{\begin{obsers}}
\newcommand{\eos}{\end{obsers}}
\newcommand{\bpf}{\begin{pf}}
\newcommand{\epf}{\end{pf}}
\newcommand{\ba}{\begin{array}}
\newcommand{\ea}{\end{array}}
\newcommand{\beq}{\begin{eqnarray}}
\newcommand{\beqq}{\begin{eqnarray*}}
\newcommand{\eeq}{\end{eqnarray}}
\newcommand{\eeqq}{\end{eqnarray*}}

\newcounter{minutes}
\divide\time by 60
\newcounter{hours}
\multiply\time by 60 \addtocounter{minutes}{-\time}
\begin{document}

\bibliographystyle{amsplain}

\title{Sphericalization with its applications in Gromov hyperbolic spaces}

\author{Qingshan Zhou}
\address{Qingshan Zhou, School of Mathematics and Big Data, Foshan University,  Foshan, Guangdong 528000, People's Republic
of China} \email{qszhou1989@163.com; q476308142@qq.com}

\author{Yaxiang Li${}^{~\mathbf{*}}$}
\address{Yaxiang Li,  Department of Mathematics, Hunan First Normal University, Changsha,
Hunan 410205, People's Republic
of China}
\email{yaxiangli@163.com}

\author{Xining Li}
\address{Xining Li,  Department of Mathematics(Zhuhai), Sun Yat-sen University, Zhuhai 519082, People's Republic
of China} \email{lixining3@mail.sysu.edu.cn}

%%%%%%%% BEGIN TIMESTAMP
\def\thefootnote{}
\footnotetext{ \texttt{\tiny File:~\jobname .tex,
          printed: \number\year-\number\month-\number\day,
          \thehours.\ifnum\theminutes<10{0}\fi\theminutes}
} \makeatletter\def\thefootnote{\@arabic\c@footnote}\makeatother
%%%%%%%% END TIMESTAMP

\date{}
 \keywords{
Sphericalization, Gromov hyperbolic space, quasihyperbolic metric, $\varphi$-uniform domain.\\
${}^{\mathbf{*}}$ Corresponding author}

\begin{abstract} In this paper, we study certain applications of sphericalization  in Gromov hyperbolic metric spaces. We first show that the doubling property regarding two classes of metrics on the Gromov boundary of hyperbolic spaces are coincided. Next, we obtain a characterization of unbounded Gromov hyperbolic domains via metric spaces sphericalization. Finally, we investigate the topological equivalence of Gromov hyperbolic $\varphi$-uniform domains between the Gromov boundary and the inner metric boundary.

\end{abstract}

\thanks{The first author was supported by NNSF of	China (No. 11901090), and by Department of Education of Guangdong Province, China (No. 2018KQNCX285). The second author  was supported by  NNSF of China  (Nos. 11601529,  11671127), and the third author was   supported by NNSF of
	China (No. 11701582).}

\maketitle{} \pagestyle{myheadings} \markboth{Sphericalization with its applications in Gromov hyperbolic spaces}{Zhou et al.}

\section{Introduction and main results}\label{sec-1}

The {\it sphericalization} of a locally compact metric space was first introduced by Bonk and Kleiner \cite{BK02} in defining a metric on the one point compactification of an unbounded space. It is a natural generalization of the deformation from the Euclidean distance on $\mathbb{R}^n$ to the chordal distance on $\mathbb{S}^n$. In \cite{WY}, Wang and Yang introduced a chordal distance on $p$-adic numbers which is an ultra-metric. Recently, the authors generalized this notation to ultra-metric spaces and  provided a new proof for a recent work of Heer \cite{Heer} concerning the  quasim\"{o}bius uniformization of Cantor set in \cite{ZLL1}.

Inspired by \cite{BK02}, Balogh and Buckley in \cite{BB} defined  a {\it flattening} transformation on a bounded metric space. It was shown in \cite{BHX} that these two conformal transformations are dual in the sense that if one starts from a bounded metric space, then performs a flattening transformation followed by a sphericalization, then the object space is bilipschitz equivalent to the original space. This duality comes from the idea that the stereographic projection between Euclidean space and the Riemann sphere can be realized as a special case of inversion. Sphericalization and flattening have a lot of applications in the area of analysis on metric spaces and asymptotic geometry, such as \cite{BB,BHX,BuSc,HSX, LS}.

Recently, Wildrick investigated the quasisymmetric parametrization of unbounded $2$-dimensional metric planes by using the sphericalization (named by {a warping process} in \cite{W}). It was shown in \cite{JJ} that two visual geodesic Gromov hyperbolic spaces are roughly quasi-isometric if and only if their Gromov boundaries are quasim\"obius equivalent by virtue of the flattening and sphericalized deformations. In \cite{M}, Mineyev studied the metric conformal structures on the idea boundaries of hyperbolic complexes via sphericalization.

In this paper, we investigate certain applications of sphericalization in Gromov hyperbolic metric spaces.  In \cite{Gr87}, Gromov observed that the essential large scale geometry of classical hyperbolic spaces $\mathbb{H}^n$ is determined by a {\it $\delta$-inequality} concerning quadruples of points, and meanwhile introduced the concept of $\delta$-hyperbolicity for general metric spaces. Since its appearance, the theory of Gromov hyperbolicity has been found numerous applications in the study of geometric group theory and geometric function theory, for instance \cite{BB03,BHK,BrHa,KLM} and the references therein.

We briefly review the theory of Gromov hyperbolic spaces. For a more complete exposition, see \cite{BrHa,BuSc} or Subsection \ref{s-g}. A geodesic metric space $X$ is called Gromov hyperbolic if there is a constant $\delta\geq 0$ such that each point on the side of any geodesic triangle is within the distance $\delta$ from some point on one of the other two sides. The Gromov boundary of $X$, denoted by $\partial_\infty X$, is defined as the set of equivalence classes of geodesic rays, with two geodesic rays $\gamma_i$, $i=1,2$, being equivalent if the Hausdorff distance $\dist_\mathcal{H}(\gamma_1,\gamma_2)<\infty$. There is a cone topology on $X\cup \partial_\infty X$ so that it is metrizable, see \cite{BrHa}. In \cite{BuSc}, Buyalo and Schroeder systematically investigated two different kinds of metrics (namely, Bourdon metric and  Hamenst\"adt metric,  for the definitions see Section \ref{sec-2}) on $\partial_\infty X$ which are respectively based at a point in $X$ and $\partial_\infty X$ (cf. \cite{Bou,Ha}).
%The first one is {\it Bourdon metric} and the second one is {\it Hamenst\"adt metric}, for the definitions see Section \ref{sec-2}.

As the first aim of this paper, we show that the doubling property of these two conformal gauge on the Gromov boundary of a Gromov hyperbolic space are coincided.

\begin{theorem}\label{thm-1}
Let $X$ be a Gromov hyperbolic space and $\partial_\infty X$ its Gromov boundary. Then $\partial_\infty X$ is doubling for any Bourdon metric if and only if it is doubling for any Hamenst\"adt metric.
\end{theorem}

The terminology used in Theorem \ref{thm-1} and in the rest of this section will be explained in Section \ref{sec-2}.

\begin{remark} In \cite{LS}, the third author and Shanmugalingam showed that the process of sphericalization preserves the Ahlfors regular and doubling measures in metric spaces. With the aid of this quantitative result, Heer \cite{Heer} proved the quasim\"obius invariance of doubling metric spaces,  which is needed in the proof of Theorem \ref{thm-1}. In this note, we  provide a quite different but direct proof for Heer's result by means of Assouad's embedding Theorem (see \cite[Theorem 8.1.1]{BS}).
\end{remark}

The doubling property of metric spaces plays an important role in the area of analysis on metric spaces. For instance,  Herron \cite{Her06} demonstrated that a Gromov hyperbolic abstract domain with doubling Gromov boundary carries a uniformizing volume growth density.  Recently, Wang and Zhou \cite{WZ} studied the equivalence of weakly quasim\"obius maps and quasim\"obius maps on doubling quasi-metric spaces.

A metric space $X$ is said to be of {\it bounded growth at some scale}, if there exist constants $r$ and $R$ with $R>r>0$, and $N\in \mathbb{N}$ such that every open ball of radius $R$ in $X$ can be covered by $N$ open balls of radius $r$. In \cite{BS}, Bonk and Schramm proved that a Gromov hyperbolic space with bounded growth at some scale embeds rough isometrically into the classical hyperbolic spaces with higher dimension. In \cite{BuSc}, Buyalo and Schroeder provided a different proof of this result. According to \cite[Theorem $9.2$]{BS} every Gromov hyperbolic geodesic space $X$ of bounded growth at some scale has a boundary $\partial_\infty X$ of finite Assouad dimension (which is equivalent to the doubling condition, see \cite{TV}). Here $\partial_\infty X$ is equipped with any visual (Bourdon) metric. Recently, Herron corroborated that a locally Ahlfors regular abstract domain is bounded growth at some scale associated to the quasihyperbolic metric, see \cite[Proposition 3.6]{Her06}. Combining these results with Theorem \ref{thm-1}  we obtain the following two corollaries; for the related definitions see \cite{BS, Her06}.

\begin{corollary}
Let $X$ be a Gromov hyperbolic geodesic metric space with bounded growth at some scale. Then the Gromov boundary equipped with any Bourdon or Hamenst\"adt metric is doubling.
\end{corollary}

\begin{corollary}Let $\Omega$ be a locally quasiconvex locally Ahlfors $Q$-regular abstract domain with a Gromov hyperbolic quasihyperbolization. Then the Gromov boundary equipped with any Bourdon or Hamenst\"adt metric is doubling.
\end{corollary}

The second goal of this paper is to explore the characterization of Gromov hyperbolic domains; here and hereafter, a Gromov hyperbolic domain always means an incomplete metric space which is Gromov hyperbolic in its quasihyperbolic metric $k$ (see Definition \ref{def-k}). In \cite{KLM}, Koskela, Lammi and Manojlovi\'{c} proved that if $(\Omega,d)$ is a bounded abstract domain, then  $(\Omega,k)$ is Gromov hyperbolic if and only if the length space $(\Omega,l_d)$ satisfies both the Gehring-Hayman condition and the ball separation condition. It is natural to ask whether this result holds for unbounded domains.  In this work, we consider the unbounded case via sphericalization and
establish the following result similar to \cite[Theorem 1.2]{KLM}.

\begin{theorem}\label{thm-2}
Let $Q>1$ and let $(X,d,\mu)$ be an Ahlfors $Q$-regular metric measure space with $(X,d)$ an annular quasiconvex, proper and geodesic space. Let $\Omega\subsetneq X$ be a domain $($an open connected set$)$, and let $l_d$ be the length metric on $\Omega$ associated to $d$. Then $(\Omega,k)$ is Gromov hyperbolic if and only if $(\Omega,l_d)$ satisfies both the  Gehring-Hayman condition and the ball separation condition.
\end{theorem}

In \cite{BHK}, Bonk, Heinonen and Koskela proved that a bounded domain in $\mathbb{R}^n$ is uniform if and only if it is Gromov hyperbolic with respect to the quasihyperbolic metric and there is a natural quasisymmetric identification between the Euclidean boundary and the Gromov boundary. Recently, Lammi  \cite{La} showed that the inner boundary of a Gromov hyperbolic domain with a suitable growth condition in the quasihyperbolic metric is homeomorphic to the Gromov boundary. Note that the quasihyperbolic boundary condition stated in \cite{La} implies the boundness of the domain. So it is natural  to ask whenever the Gromov boundary and the inner metric boundary of an unbounded Gromov hyperbolic domain are homeomorphic.

Motivated by these considerations, we focus on the study of the boundary behavior of Gromov hyperbolic domains, particularly, for unbounded domains. It was shown in \cite{BHK, GO} that there is a characterization of uniform domains (see Definition \ref{def-u}) in terms of two metrics, the quasihyperbolic metric $k$ and the distance ratio metric $j$ (see (\ref{def-j})). That is, a domain $D$ is uniform in $\mathbb{R}^n$ if and only if there exist constants $c_1\ge1$ and $c_2\geq 0$ such that for all $x,y\in D$,
\be\label{l-1} k_D(x,y)\le c_1 j_D(x,y) +c_2.\ee
Subsequently, Vuorinen observed that the additive constant $c_2$ on the right hand side of  (\ref{l-1}) can be chosen to be $0$.  This observation leads to the definition of $\varphi$-{\it uniform domains} introduced in \cite{Vu85}.

\begin{definition}\label{def-v} Let $X$ be a rectifiably connected, locally compact and complete metric space, and let $D\subsetneq X$ be a domain with $d_D(x)=\dist(x,\partial D)$ for all $x\in D$. Let $\varphi:[0,\infty)\to [0,\infty)$ be a homeomorphism. We say that $D$ is $\varphi$-{\it uniform} if for all $x$, $y$ in $D$,
$$k_D(x,y)\leq \varphi(r_D(x,y))\;\;\;\;\;\mbox{where}\;\;\;\;\;r_D(x,y)=\frac{d(x,y)}{d_D(x)\wedge d_D(y)}.$$
Here and hereafter, $r\wedge s=\min\{r,s\}$ for all $r,s\in \mathbb{R}$.
\end{definition}

In \cite{Vai-4}, V\"ais\"al\"a  has also investigated this class of domains, and pointed out that these two classes of domains are the same provided $\varphi$ is a slow function (i.e. a function $\varphi$ satisfying $\lim_{t\to\infty} \varphi(t)/t=0$). We remark that every convex domain is $\varphi$-uniform with $\varphi(t)=t$. However, in general, convex domains need not be uniform. For example, $D=\{(x_1,x_2)\in \mathbb{R}^2:0<x_2<1\}$ is $\varphi$-uniform with $\varphi(t)=t$, but it is not uniform.

The third purpose of this paper is to study whether or not the $\varphi$-uniformity condition is sufficient for the homeomorphism equivalence between the Gromov boundary and metric boundary of Gromov hyperbolic domain. Our result in this direction is as follows.

\begin{theorem}\label{thm-3}
Let $Q>1$ and let $(\Omega,d,\mu)$ be a locally compact, $c$-quasiconvex, Ahlfors $Q$-regular incomplete metric measure space. Assume that $(\Omega,k)$ is roughly starlike Gromov hyperbolic and $(\Omega,d)$ is $\varphi$-uniform, where $k$ is the quasihyperbolic metric of $\Omega$.
\begin{enumerate}
  \item If $\Omega$ is bounded and $\int_1^\infty \frac{dt}{{\varphi^{-1}(t)}}<\infty$, then the Gromov boundary and the metric boundary of $\Omega$ are homeomorphic;
  \item  If $\Omega$ is unbounded and $\int_1^\infty \frac{dt}{\sqrt{\varphi^{-1}(t)}}<\infty$, then the Gromov boundary of $\Omega$ and $\partial \Omega\cup\{\infty\}$ are homeomorphic. Here and hereafter, $\overline{\Omega}\cup \{\infty\}$ is the one point compactification of the space $(\Omega,d)$ and $\partial \Omega\cup\{\infty\}=\overline{\Omega}\cup \{\infty\}\setminus \Omega$.
\end{enumerate}

\end{theorem}

Note that Theorem \ref{thm-3} is new even for  Gromov hyperbolic domains in $\mathbb{R}^n$. It follows from \cite[Theorem 1.11]{BHK} that (inner) uniform domains in $\mathbb{R}^n$ are Gromov hyperbolic. However,  $\varphi$-uniform domains need not be Gromov hyperbolic. Let $G=\{(x_1,x_2,x_3)\in \mathbb{R}^3:0<x_3<1\}$. Thus $G$ is a convex domain which is $\varphi$-uniform with $\varphi(t)=t$. Moreover, it is not difficult to check that $G$ is $LLC_2$ (see \cite[Chapter 7]{BHK}) but not $c$-uniform for any $c\geq 1$. Hence we see from \cite[Proposition 7.12]{BHK} that $G$ is not $\delta$-hyperbolic for any $\delta\geq 0$.

We note that a Gromov hyperbolic $\varphi$-uniform  domain in  $\mathbb{R}^n$ is not necessarily uniform.   For instance, $\Omega=\{(x_1,x_2)\in \mathbb{R}^2:0<x_2<1\}$ is a plane simply connected convex domain, and therefore it is a Gromov hyperbolic $\varphi$-uniform  domain. But it is not uniform. Therefore, Theorem \ref{thm-3} is a generalization of the results in \cite{BHK, La}. Moreover, the unbounded case is also in our considerations by using the sphericalized transformation.

The organization of this paper is as follows. In Section \ref{sec-2}, we recall some definitions and preliminary results. In Section \ref{sec-3}, we will prove the main results.

\section{Preliminary and Notations}\label{sec-2}
\subsection{Metric geometry}
In this paper, we always use $(X,d)$, $(X',d')$, $(Y, d)$  etc to denote metric spaces.
For $(X,d)$, its metric completion and metric boundary are denoted by $\overline{X}$ and $\partial X:= \overline{X}\setminus X$, respectively. A domain $D\subset X$ is an open connected set.

%For a domain $\Omega$ in $X$ or $X'$  and $u\in \Omega$, if $\Omega\not=X$ or $X'$, $d_{\Omega}(u)$ denotes the distance from $u$ to the boundary $\partial \Omega$ of $\Omega$ with respect to the metric $d$ or $d'$. If there is no risk of confusion, we will omit the subscript $\Omega$ in the symbol.

The open (resp. closed) metric ball with center $x\in X$ and radius $r>0$ is denoted by
$${\mathbb{B}}(x,r)=\{z\in X:\; d(z,x)<r\}\;\;(\mbox{resp.}\;\; \overline{{\mathbb{B}}}(x,r)=\{z\in X:\; d(z,x)\leq r\}).$$ We say that $X$ is {\it incomplete} if $\partial X\neq \emptyset$. $X$ is called {\it proper} if all its closed balls are compact.

%It is known that every proper topological vector space has finite dimension (cf. \cite[Theorem 1.22]{rudin-1}).

By a curve, we mean a continuous function $\gamma:$ $I=[a,b]\to X$. The length of $\gamma$ is denoted by
$$\ell(\gamma)=\sup\Big\{\sum_{i=1}^{n}d\big(\gamma(t_i),\gamma(t_{i-1})\big)\Big\},$$
where the supremum is taken over all partitions $a=t_0<t_1<t_2<\ldots<t_n=b$. The curve is rectifiable if $\ell(\gamma)<\infty$.
 We also denote the image $\gamma(I)$ of $\gamma$ by $\gamma$, and the subcurve of $\gamma$ with endpoints $x$ and $y$ by $\gamma[x,y]$. A curve $\gamma$ is called {\it rectifiable}, if the length $\ell_d(\gamma)<\infty$.
A metric space $(X, d)$ is called {\it rectifiably connected} if every pair of points in $X$ can be joined with a rectifiable curve $\gamma$.

%If $\gamma[x,y]$ is rectifiable and $\ell(\gamma[x,y])=d(x,y)$, then $\gamma[x,y]$ is called a {\it geodesic} connecting $x$ and $y$ with respect to the metric $d$.

Suppose the curve $\gamma$ is rectifiable. The length function $s_{\gamma}$: $[a,b]\to [0, \ell(\gamma)]$ is defined by $$s_{\gamma}(t)=\ell(\gamma[a,t]).$$
Then there is a unique curve $\gamma_s:$ $[0, \ell(\gamma)]\to X$ such that $$\gamma=\gamma_s\circ s_{\gamma}.$$ Obviously,
$\ell(\gamma_s[0,t])=t$ for each $t\in [0, \ell(\gamma)]$. The curve $\gamma_s$ is called the {\it arc-length parametrization} of $\gamma$.

For a rectifiable curve $\gamma$ in $X$, the line integral over $\gamma$ of each Borel function $\varrho:$ $X\to [0, \infty)$ is defined as follows:
$$\int_{\gamma}\varrho ds=\int_{0}^{\ell(\gamma)}\varrho\circ \gamma_s(t) dt.$$
%Let $(X, d)$ be a metric space, we denote the open (resp. closed) metric ball with center $x\in X$ and radius $r>0$ by
%$$B(x,r)=\{z\in X:\; d(z,x)<r\}\;\;(\mbox{resp.}\;\; \overline{B}(x,r)=\{z\in X:\; d(z,x)\leq r\}),$$
%Let $\overline{X}$ be the metric completion of $X$, and $\partial X= \overline{X}\setminus X$ denote the metric boundary of $X$. We say that $X$ is {\it incomplete} if $\partial X\neq \emptyset$. $X$ is called {\it proper} if all closed balls are compact.

%By a curve, we mean a continuous function $\gamma:$ $I=[a,b]\to X$. If $\gamma$ is an embedding of $I$,
%it is called an {\it arc}. The image set $\gamma(I)$ of $\gamma$ is also denoted by $\gamma$. A curve $\gamma$ is called {\it rectifiable}, if the length $\ell_d(\gamma)<\infty$.
%A metric space $(X, d)$ is called {\it rectifiably connected} if every pair of points in $X$ can be joined with a rectifiable
%curve $\gamma$.

%The length function associated with a rectifiable curve $\gamma$: $[a,b]\to X$ is $z_{\gamma}$: $[a,b]\to [0, \ell(\gamma)]$, given by
%$z_{\gamma}(t)=\ell(\gamma|_{[a,t]})$.
%For any rectifiable curve $\gamma:$ $[a,b]\to X$, there is a unique map $\gamma_s:$ $[0, \ell(\gamma)]\to X$ such that $\gamma=\gamma_s\circ z_{\gamma}$. Thus
%$\ell(\gamma_s|_{[0,t]})=t$ for $t\in [0, \ell(\gamma)]$. The curve $\gamma_s$ is called the {\it arclength parametrization} of $\gamma$.

%For a rectifiable curve $\gamma$ in $X$, the line integral over $\gamma$ of each Borel function $\varrho:$ $X\to [0, \infty)$ is
%$$\int_{\gamma}\varrho ds=\int_{0}^{\ell(\gamma)}\varrho\circ \gamma_s(t) dt.$$

For $c\geq 1$, a curve $\gamma\subset X$, with endpoints $x,y$, is {\it$c$-quasiconvex},
$c \geq 1$, if its length is at most $c$ times the distance between its endpoints; i.e., if $\gamma$
satisfies $$\ell(\gamma)\leq cd(x,y).$$
$X$ is {\it $c$-quasiconvex} if each pair of points can be joined by a $c$-quasiconvex curve.

Let $c \geq 1$ and $0 < \lambda \leq 1/2$. An incomplete metric space $(D, d)$ is said to be {\it locally $(\lambda,c)$-quasiconvex,} if for all $x\in D$, each pair of points in ${{\mathbb{B}}}(x, \lambda d_D(x))$ can be joined with a $c$-quasiconvex curve. A {\it geodesic} $\gamma$ joining $x$ to $y$ in $X$ is a map $\gamma:I=[0,l]\to X$ from an interval $I$ to $X$ such that $\gamma(0)=x$, $\gamma(l)=y$ and
$$\;\;\;\;\;\;\;\;d(\gamma(t),\gamma(t'))=|t-t'|\;\;\;\;\;\;\;\;\;\;\;\;\;\;\;\;\;\;\;\;\mbox{for all}\;\;t,t'\in I.$$
If $I=[0,\infty)$ or $\mathbb{R}$, then $\gamma$ is called a {\it geodesic ray} or a {\it geodesic line}.  A metric space $X$ is said to be {\it geodesic} if every pair of points can be joined by a geodesic arc.

Every rectifiably connected metric space $(X, d)$ admits a length (or intrinsic) metric, its  so-called length distance given by
$$\ell_d(x, y) = \inf\ell_d(\gamma),$$
where the infimum is taken over all rectifiable curves $\gamma$
joining $x$ to $y$ in $X$.

\begin{definition}
Let $X$ be a metric space with $p\in X$, and let $c\geq 1$. We say that $X$ is $c$-{\it annular quasiconvex with respect to $p$}, if for all $r>0$, every pair of points in the annular $\overline{{\mathbb{B}}}(p,2r)\setminus {\mathbb{B}}(p,r)$ can be joined by a $c$-quasiconvex curve lying in $\overline{{\mathbb{B}}}(p,2cr)\setminus {\mathbb{B}}(p,r/c)$. $X$ is called $c$-{\it annular quasiconvex} if it is $c$-annular quasiconvex at each point.
\end{definition}

%In the remaining part of this subsection, we always assume that

%\emph{$(X,d)$ is a rectifiably connected, locally compact and complete metric space, and that $D\subsetneq X$ is a domain.}

\begin{definition}\label{def-k}  Let $(X,d)$ be a rectifiably connected, locally compact and complete metric space, and let $D\subsetneq X$ is a domain. The {\it quasihyperbolic length} of a curve $\gamma \subset D$ is defined as
$$\ell_{k}(\gamma)=\ell_{k_D}(\gamma)=\int_{\gamma}\frac{|dz|}{d_D(z)}.$$
For any $x,y\in D$, the {\it quasihyperbolic distance}
$k(x,y)$ between $x$ and $y$ is defined by
$$k(x,y)=k_D(x,y)=\inf\ell_{k}(\gamma),
$$
where the infimum is taken over all rectifiable curves $\gamma$
joining $x$ to $y$ in $D$.
\end{definition}
We remark that the resulting space $(D,k)$ is complete, proper and geodesic whenever $D$ is locally quasiconvex (cf. \cite[Proposition $2.8$]{BHK}). There is another useful metric in the study of geometric function theory. We recall the definition of distance ratio metric $j$ as follows:
\be\label{def-j} j_D(x,y)=\log\Big(1+\frac{d(x,y)}{d_D(x)\wedge d_D(y)}\Big).\ee

\begin{definition} Let $(X,d)$ be a rectifiably connected, locally compact and complete metric space, and let $D\subsetneq X$ is a domain. Let $C_{gh}\geq 1$ be a constant. We say that $D$ satisfies the {\it $C_{gh}$-Gehring-Hayman inequality}, if for all $x$, $y$ in $D$ and for each
quasihyperbolic geodesic $\gamma$ joining $x$ and $y$ in $D$, we have
$$\ell(\gamma)\leq C_{gh}\ell(\beta_{x,y}),$$
where $\beta_{x,y}$ is any other curve joining $x$ and $y$ in $D$. In other words,  quasihyperbolic geodesics are essentially the shortest curves in $D$.
\end{definition}

\begin{definition} Let $(X,d)$ be a rectifiably connected, locally compact and complete metric space, and let $D\subsetneq X$ is a domain.  Let  $C_{bs}\geq 1$ be a constant. We say that $D$ satisfies the {\it $C_{bs}$-ball separation condition}, if for all $x$, $y$ in $D$ and for each
quasihyperbolic geodesic $\gamma$ joining $x$ and $y$ in $D$, we have for every $z\in \gamma$,
$${\mathbb{B}}(z,C_{bs}d_D(z)) \cap \beta_{x,y} \not=\emptyset ,$$
where $\beta_{x,y}$ is any other curve joining $x$ and $y$ in $D$.
\end{definition}

\begin{definition}\label{def-u} Let $(X,d)$ be a rectifiably connected, locally compact and complete metric space, and let $D\subsetneq X$ is a domain. We say that $D$ is $c$-{\it uniform} provided there exists a constant $c$
with the property that each pair of points $x$, $y$ in $D$ can
be joined by a rectifiable arc $\gamma$ in $D$ satisfying
\begin{enumerate}
\item $\min\{\ell(\gamma[x,z]),\ell(\gamma[z,y])\}\leq c\,d_D(z)$ for all $z\in \gamma$, and
\item $\ell(\gamma)\leq c\,d(x,y)$,
\end{enumerate}
\noindent where $\gamma[x,z]$ the part of $\gamma$ between $x$ and $z$.
\end{definition}

\subsection{Mappings on metric spaces}
%%%%%%%%%%%%%%%%%%%
In this part, we recall certain definitions for mappings between metric spaces. Here primes always denote the images of points under $f$, for example, $x'=f(x)$.

A quadruple in $X$ is an ordered sequence $Q = (x,y,z,w)$ of four distinct points
in $X$. The {\it cross ratio} of $Q$ is defined to be the number
$$r(x,y,z,w) = \frac{d(x,z)d(y,w)}{d(x,y)d(z,w)}.$$

Observe that the definition is extended in the well known manner to the case
where one of the points is $\infty$. For example,
$$r(x,y,z,\infty) = \frac{d(x,z)}{d(x,y)}.$$

Suppose that $\eta$ and $\theta$ are homeomorphisms from $[0, \infty)$ to $[0, \infty)$, and that $f:$ $(X_1,d_1)\to (X_2,d_2)$ is an embedding between two metric spaces. Then
\begin{enumerate}
  \item we call that $f$ is {\it $L$-bilipschitz} for some $L\geq 1$ if for all $x,y\in X_1,$
   $$L^{-1}d_1(x,y)\le d_2(x',y')\le Ld_1(x,y).$$
  \item $f$ is said to be {\it $\eta$-quasisymmetric} if for all $x$, $y$ and $z$ in $X_1$,
   $$d_1(x,y)\leq t d_1(x,z)\;\;\;\;\;\;\mbox{implies that}\;\;\;\;\;\;d_2(x',y')\leq \eta(t) d_2(x',z').$$
  \item $f$ is called {\it $\theta$-quasim\"obius} if  for all $x,y,z,w$ in $X_1$,
   $$r(x,y,z,w)\leq t\;\;\;\;\;\;\mbox{implies that}\;\;\;\;\;\;r(x',y',z',w')\leq \theta(t).$$
\end{enumerate}

For a set $A$ in a metric space $X$, it is called {\it $c$-cobounded} in $X$ for $c\geq 0$, if  every point $x\in X$ has distance at most $c$ from $A$. If $A$ is $c$-cobounded for some $c\geq 0$, then we say that $A$ is {\it cobounded} in $X$ \cite{BS}.

\begin{definition} Let $\lambda\geq 1$ and $c\geq 0$. A mapping  $f:$ $(X_1,d_1)\to (X_2,d_2)$  is said to be a {\it roughly $(\lambda, c)$-quasi-isometry}, if $f(X_1)$ is $c$-cobounded in $X_2$ and  for all $x,y\in X$,
$$\frac{1}{\lambda}d_1(x,y)-c\leq d_{2}(x',y')\leq \lambda d_1(x,y)+c.$$
\end{definition}

\subsection{Sphericalization of metric measure spaces} In this subsection, we recall some materials concerning metric measure spaces and sphericalization,  for which we refer to standard references \cite{BB03,HSX,LS}.

Let $(X, d)$ be a metric space. $X$ is said to be {\it doubling} if there is a constant $C$ such that every (metric) ball ${\mathbb{B}}$ in $X$ can be covered with at most $C$ balls of half the radius of ${\mathbb{B}}$. A positive Borel measure $\mu$ on $X$ is {\it doubling} if there is a constant $C_\mu$ such that
$$\mu({\mathbb{B}}(x,2r))\leq C_\mu \mu({\mathbb{B}}(x,r))$$
for all $x\in X$ and $r>0$. Moreover, $X$ is said to be {\it Ahlfors $Q$-regular} if it admits a positive Borel measure $\mu$ such that
$$C^{-1}R^Q\leq \mu({{\mathbb{B}}}(x,R)) \leq C R^Q$$
for all $x\in X$ and $0<R< \diam(X)$ (it is possible that the diameter of $X$ satisfies $\diam(X)=\infty$), where $C\geq 1$ and $Q>0$ are constants. Note that Ahlfors regular spaces are necessarily doubling. For instance, the Euclidean space $\mathbb{R}^n$ with Lebesgue measure satisfies the Ahlfors $n$-regularity.

Given an unbounded locally compact metric space $(X,d)$ and a non-isolated point $a\in X$, we consider the one point compactification $\dot{X}=X\cup \{\infty\}$ and define $d_a:\dot{X}\times \dot{X}\to[0,\infty)$ as follows
$$d_a(x,y)=d_a(y,x)=\begin{cases}
\displaystyle\; \frac{d(x,y)}{[1+d(x,a)][1+d(y,a)]},\;
\;\;\;\;\mbox{if}\;\;x,y\in X,\\
\displaystyle\;\;\;\;\;\frac{1}{1+d(x,a)},\;\;\;\;
\;\;\;\;\;\;\;\;\;\;\;\;\;\;\;\;\; \mbox{if}\;\;y=\infty\; \mbox{and}\;x\in X,\\
\displaystyle\;\;\;\;\;\;\;\;\;\;0,
\;\;\;\;\;\;\;\;\;\;\;\;\;\;\;\;\;\;\;\;\;\;\;\;\;\;\;\;\;\;\;\mbox{if}\;\; x=\infty=y.
\end{cases}
$$

In general, $d_a$ is not a metric on $X$, but a quasimetric. However, there is a standard procedure, known as {\it chain construction}, to construct a metric from a quasimetric as follows. Define
$$\widehat{d}_a(x,y):=\inf \Big\{ \sum_{j=0}^k d_a(x_j,x_{j+1})\Big\},$$
where the infimum is taken over all finite sequences $x=x_0,x_1,...,x_k,x_{k+1}=y$ from $\dot{X}$.

Then $(\dot{X},\widehat{d}_a)$ is a metric space and it is said to be the {\it sphericalization} of $(X,d)$ associated to the point $a\in \dot{X}$. Moreover, by \cite[Lemmma 2.2.5]{BuSc}, we have for all $x,y\in \dot{X}$
\be\label{z-1.1} \frac{1}{4}d_a(x,y)\leq \widehat{d}_a(x,y)\leq d_a(x,y).\ee

In the case that $(X,d)$ is a rectifiably connected unbounded metric space, we define a Borel function $\rho_a:X\to [0,\infty)$ by
$$\rho_a(x)=\frac{1}{[1+d(a,x)]^2}.$$
A similar argument as \cite[4.1 and (4.2)]{BHX}, we obtain that for any rectifiable curve $\gamma$ joining $x$ and $y$,
\be\label{z-1.2} \ell_{\widehat{d}_a}(\gamma)=\int_\gamma \rho_a(z)|dz|.\ee

Suppose that $(X,d)$ is  a proper space equipped with a Borel-regular measure $\mu$ such that the measures of non-empty open bounded sets are positive and finite. We define the {\it spherical measure} $\mu_a$ associated to the point $a\in X$ as
$$\mu_a(A)=\int_{A\setminus \{\infty\}}\frac{1}{\mu({\mathbb{B}}(a,1+d(a,z)))^2}d\mu(z).$$

\subsection{Gromov hyperbolic spaces}\label{s-g} In this subsection, we give some basic information about Gromov hyperbolic spaces, for which we refer to standard references \cite{BS,BrHa,BuSc,Gr87,Vai-0}.

We begin with the definition of $\delta$-hyperbolic spaces. We say that a metric space $(X,d)$ is {\it Gromov hyperbolic}, if there is a constant $\delta\geq 0$ such that it satisfies the following $\delta$-inequality
$$(x|y)_w\geq \min\{(x|z)_w,(z|y)_w\}-\delta$$
for all $x,y,z,w\in X$, where $(x|y)_w$ is the {\it Gromov product} with respect to $w$ defined by
$$(x|y)_w=\frac{1}{2}[d(x,w)+d(y,w)-d(x,y)].$$

\begin{definition} Suppose that $(X, d)$ is a Gromov $\delta$-hyperbolic metric space for some constant $\delta\geq 0$.
\begin{enumerate}
\item
A sequence $\{x_i\}$ in $X$ is called a {\it Gromov sequence} if $(x_i|x_j)_w\rightarrow \infty$ as $i,$ $j\rightarrow \infty.$
\item
Two such sequences $\{x_i\}$ and $\{y_j\}$ are said to be {\it equivalent} if $(x_i|y_i)_w\rightarrow \infty$.
\item
The {\it Gromov boundary} or {\it the boundary at infinity} $\partial_\infty X$ of $X$ is defined to be the set of all equivalence classes.
\item
For $a\in X$ and $\eta\in \partial_\infty X$, the Gromov product $(a|\eta)_w$ of $a$ and $\eta$ is defined by
$$(a|\eta)_w= \inf \big\{ \liminf_{i\rightarrow \infty}(a|b_i)_w:\; \{b_i\}\in \eta\big\}.$$
\item
For $\xi,$ $\eta\in \partial_\infty X$, the Gromov product $(\xi|\eta)_w$ of $\xi$ and $\eta$ is defined by
$$(\xi|\eta)_w= \inf \big\{ \liminf_{i\rightarrow \infty}(a_i|b_i)_w:\; \{a_i\}\in \xi\;\;{\rm and}\;\; \{b_i\}\in \eta\big\}.$$
\end{enumerate}
\end{definition}

Let $X$ be a proper, geodesic $\delta$-hyperbolic space, and let $w\in X$. We say that $X$ is {\it $K$-roughly starlike} with respect to  $w$ if for each $x\in X$ there is some point $\eta\in\partial_\infty X$ and a geodesic ray $\gamma=[w,\eta]$ emanating from $w$ to $\eta$ with
$$\dist(x,\gamma)\leq K.$$
This concept was introduced by Bonk, Heinonen and Koskela \cite{BHK}, see also \cite{BS}. They proved that bounded uniform spaces and every Gromov hyperbolic domain  in $\IR^n$ are roughly starlike.

For $0<\varepsilon<\min\{1,\frac{1}{5\delta}\}$, define
$$\rho_{w,\varepsilon}(\xi,\zeta) =e^{-\varepsilon(\xi|\zeta)_w}$$
for all $\xi,\zeta$ in the Gromov boundary $\partial_\infty X$ of $X$ with convention $e^{-\infty}=0$.

We now define
$$d_{w,\varepsilon}(\xi,\zeta):=\inf \Big\{\sum_{i=1}^{n} \rho_{w,\varepsilon} (\xi_{i-1},\xi_i):n\geq 1,\xi=\xi_0,\xi_1,...,\xi_n=\zeta\in \partial_\infty X\Big\}.$$
Then $(\partial_\infty X,d_{w,\varepsilon})$ is a metric space with
$$\rho_{w,\varepsilon}/2\leq d_{w,\varepsilon}\leq \rho_{w,\varepsilon},$$
and we call $d_{w,\varepsilon}$ the {\it Bourdon metric} of $\partial_\infty X$ based at $w$ with the parameter $\varepsilon$.

Following \cite{BuSc}, we say that $b:X\to \mathbb{R}$ is a {\it Busemann function} based at $\xi$, denoted by $b\in \mathcal{B}(\xi)$, if for some $w\in X$, we have
$$b(x)=b_{\xi,w}(x)=b_\xi(x,w)=(\xi|w)_x-(\xi|x)_w\;\;\;\;\;\;\;\;\;\;\;\;\mbox{for}\;x\in X.$$

We next define the Gromov product of $x,y\in X$ based at the Busemann function $b=b_{\xi,w}\in \mathcal{B}(\xi)$ by
$$(x|y)_b=\frac{1}{2}(b(x)+b(y)-d(x,y)).$$
Similarly, for $x\in X$ and $\eta\in \partial_\infty X\setminus\{\xi\}$, the Gromov product $(x|\eta)_b$ of $x$ and $\eta$ is defined by
$$(x|\eta)_b= \inf \big\{ \liminf_{i\rightarrow \infty}(x|z_i)_b:\; \{z_i\}\in \eta\big\}.$$
For points $\xi_1,\xi_2\in \partial_\infty X\setminus\{\xi\}$, we define their Gromov product based at $b$ by
$$(\xi_1|\xi_2)_b=\inf\big\{\liminf_{i\to\infty} (x_i|y_i)_b: \{x_i\}\in\xi_1 , \{y_i\}\in\xi_2\}.$$

According to \cite[$(3.2)$ and Example $3.2.1$]{BuSc}, we know that
\be\label{z0} (x|y)_b\doteq_{10\delta} (x|y)_{w,\xi}:=(x|y)_w-(\xi|x)_w-(\xi|y)_w,\ee
where $x\doteq_{a}y$ means that $|x-y|\le a$ for some real number $a\geq 0$.

It follows from \cite[Proporsition $3.2.3$]{BuSc} that $(x|y)_b,(y|z)_b,(x|z)_b$ form a $22\delta$-triple for every $x,y,z\in X$.

Now we recall the definition of the {\it Hamenst\"adt metric} of $\partial_\infty X$ based at $\xi$ or a Busemann function $b=b_{\xi,w}\in \mathcal{B}(\xi)$. For $\varepsilon>0$ with $e^{22\varepsilon\delta}\leq 2$, define
$$\rho_{b,\varepsilon}(\xi_1,\xi_2)= e^{-\varepsilon(\xi_1|\xi_2)_b}\;\;\;\;\;\;\;\mbox{for all}\;\xi_1,\xi_2\in \partial_\infty X.$$
Then for $i=1,2,3$ with $\xi_i\in \partial_\infty X\setminus\{\xi\}$, we have
$$\rho_{b,\varepsilon}(\xi_1,\xi_2)\leq e^{22\varepsilon\delta} \max\{\rho_{b,\varepsilon}(\xi_1,\xi_3),\rho_{b,\varepsilon}(\xi_3,\xi_2)\}.$$
That is, $\rho_{b,\varepsilon}$ is a $K'$-quasi-metric on $\partial_\infty X\setminus\{\xi\}$ with $K'=e^{22\varepsilon\delta}\leq 2$. We now define
$$\sigma_{b,\varepsilon}(x,y):=\inf \Big\{\sum_{i=1}^{n} \rho_{b,\varepsilon} (x_{i-1},x_i):n\geq 1,x=x_0,x_1,...,x_n=y\in\partial_\infty X\setminus\{\xi\} \Big\}.$$
By \cite[Lemma $3.3.3$]{BuSc}, we see that $(\partial_\infty X\setminus\{\xi\}, \sigma_{b,\varepsilon})$ is a metric space with
$$\rho_{b,\varepsilon}/2\leq \sigma_{b,\varepsilon}\leq \rho_{b,\varepsilon}.$$
Then $\sigma_{b,\varepsilon}$ is called the {\it Hamenst\"adt metric} on the punctured space $\partial_\infty X\setminus\{\xi\}$ based at $\xi$ with parameter $\varepsilon$.

To conclude this part, we note that $\partial_\infty X$ equipped with any Bourdon metric is bounded. However, the punctured space $\partial_\infty X\setminus\{\xi\}$ equipped with any Hamenst\"adt metric $\sigma_{b,\varepsilon}$ is unbounded.

%%%%%%%%%%%%%%%%%

\section{Proofs of the main results}\label{sec-3}

\subsection{} In this subsection, we shall give the proof of Theorem \ref{thm-1}. To this end, we introduce some necessary results. The first one is known as Assouad's embedding Theorem.

\begin{theorem}\label{Thm-2} $($\cite[Theorem $8.1.1$]{BuSc}$)$
Let $(Z,d)$ be a doubling metric space. Then for every $s\in(0,1)$ there is a bilipschitz embedding $\varphi:(Z,d^s) \to \mathbb{R}^N$, where $N\in \mathbb{N}$ depends only on $s$ and the doubling constant of the metric $d$.
\end{theorem}

Secondly, the following auxiliary result concerns the quasim\"obius invariance of doubling metric spaces. We note that this lemma has been proved by Heer \cite{Heer}. His proof was based on a recent work of the third author and Shanmugalingam \cite{LS}. However, our arguments are quite different and based on Assouad's embedding Theorem. We think it is interesting and thus present a proof here.

\begin{lemma}\label{lem-2}
Let $(X,d)$ and $(X',d')$ be  metric spaces, and let $f:X\to X'$ be a quasim\"{o}bius embedding. If $X'$ is a doubling metric space, then $X$ is doubling.
\end{lemma}

\begin{proof} Without loss of generality, we may assume that $X'=f(X)$ and $f$ is a homeomorphism because the subspace of a doubling metric space is also doubling, see \cite[Remark 2.8]{TV}. Then according to Theorem \ref{Thm-2}, we see that there is a quasisymmetric embedding
$$\varphi:(X',d') \to Z:=\varphi(X')\subset \mathbb{R}^N,$$
where $N\in \mathbb{N}$ depends only on the doubling constant of the metric $d'$. Moreover, we obtain a quasim\"{o}bius homeomorphism
$$g=\varphi\circ f:X\to Z.$$
In order to show that $X$ is doubling, we consider two cases.

\begin{case} $X$ is unbounded.\end{case}
If $g(x)\to\infty$ as $x\to \infty$, by \cite[Theorem $3.10$]{Vai-5}, we see that $g$ is quasisymmetric and thus the claim follows from \cite[Theorem 2.10]{TV}.

If $g(x)\to p\neq\infty$ as $x\to \infty$,  we may assume without loss of generality that $g(\infty)=0\in Z$. Let $$u(x)=\frac{x}{|x|^2}$$
be the reflection about the unit sphere centered at the origin in $\mathbb{R}^N\cup \{\infty\}$. Clearly, $u$ is a M\"{o}bius transformation. It follows that $u\circ g:X\to u(Z)$ is quasim\"{o}bius with $u\circ g(x)\to \infty$ as $x\to \infty$. Again by \cite[Theorem 3.10]{Vai-5} we see that $u\circ g$ is quasisymmetric. Therefore, it follows from \cite[Theorem 2.10]{TV} that $X$ is a doubling metric space.

\begin{case} $X$ is bounded. \end{case}
In this case, we consider the spherical metric $\sigma$ on $\mathbb{R}^N\cup \{\infty\}$ which is determined by the length element
$$|dz|_\sigma=\frac{2|dz|}{1+|z|^2},$$
where $|dz|$ is the Euclidean length element and $|z|$ is the Euclidean norm of a point $z\in \mathbb{R}^N$. Then $g:X\to (Z,\sigma)$ is a quasim\"{o}bius map between two bounded metric spaces, which is actually quasisymmetric. By \cite[Theorem 2.10]{TV}, we see that $X$ is doubling.
\end{proof}

%\br Let us remark that this lemma was proved by Heer \cite{Heer} in showing quasim\"obius uniformization of symbolic Cantor sets. His proof was based on a recent work of Li and Shanmugalingam \cite{LS}. However, our arguments are quite different.
%\er

Next, we show that the Gromov boundary of a hyperbolic space equipped with any Bourdon metric and Hamenst\"adt metic are quasim\"obius equivalent.

\begin{lemma}\label{lem-3} Let $X$ be a Gromov $\delta$-hyperbolic space and $\partial_\infty X$ its Gromov boundary. Then the identity map $(\partial_\infty X\setminus\{\xi\}, d_{w,\varepsilon})\to (\partial_\infty X\setminus\{\xi\}, \sigma_{b,\varepsilon'})$ is $\theta$-quasim\"obius with $\theta$ depending only on $\delta,\varepsilon$ and $\varepsilon'$, where $d_{w,\varepsilon}$ is the Bourdon metric based at $w\in X$ with parameter $\varepsilon>0$ and
$\sigma_{b,\varepsilon'}$ is the H\"amenstadt metric based at the Busemann function $b=b_{\xi,o}$ with parameter $\varepsilon'>0$ for $o\in X$ and $\xi\in \partial_\infty X$.
\end{lemma}
\begin{proof}
For any distinct points $x_i\in \partial_\infty X\setminus\{\xi\}$, $i=1,2,3,4$,  by \cite[Lemmas 2.2.2 and 3.2.4]{BuSc} we may assume that all of them lie in $X$.  Thus by $(\ref{z0})$, we obtain
$$(x_1|x_2)_w+(x_3|x_4)_w-(x_1|x_3)_w-(x_2|x_4)_w$$
\beq\nonumber
&=& (x_1|x_2)_o+(x_3|x_4)_o-(x_1|x_3)_o-(x_2|x_4)_o
\\ \nonumber &\doteq_{40\delta}& (x_1|x_2)_b+(x_3|x_4)_b-(x_1|x_3)_b-(x_2|x_4)_b.
\eeq
Therefore, we have
\beq\nonumber
\frac{\sigma_{b,\varepsilon'}(x_1,x_3)\sigma_{b,\varepsilon'}(x_2,x_4)}{\sigma_{b,\varepsilon'}(x_1,x_2)\sigma_{b,\varepsilon'}(x_3,x_4)}&\leq& 4e^{-\varepsilon'[(x_1|x_2)_b+(x_3|x_4)_b-(x_1|x_3)_b-(x_2|x_4)_b]}
\\ \nonumber &\leq& 4e^{-40\varepsilon'\delta}e^{-\varepsilon'[(x_1|x_2)_w+(x_3|x_4)_w-(x_1|x_3)_w-(x_2|x_4)_w]}
\\ \nonumber &\leq& 4e^{-40\varepsilon'\delta}4^{\varepsilon'/\varepsilon}
\Big[\frac{d_{w,\varepsilon}(x_1,x_3)d_{w,\varepsilon}(x_2,x_4)}{d_{w,\varepsilon}(x_1,x_2)d_{w,\varepsilon}(x_3,x_4)}\Big]^{\varepsilon'/\varepsilon}.
\eeq
This proves Lemma \ref{lem-3}.
\end{proof}

\emph{Proof of Theorem \ref{thm-1}.} This follows from Lemmas \ref{lem-2} and \ref{lem-3}.\qed

\subsection{} In this subsection,  we will prove Theorem \ref{thm-2} with the aid of the following auxiliary results. It seems that Lemma \ref{lem-4} below is well known, but we  failed to find a reference for its proof. As we will use this to prove Theorem \ref{thm-2}, we give a proof here.

\begin{lemma}\label{lem-4}  Let $f:(X,d)\to (Y,d')$ be a rough $(\lambda,c)$-quasi-isometry between two proper geodesic Gromov hyperbolic spaces. If $X$ is $K$-roughly starlike with respect to some point $w\in X$, then $Y$ is $K'$-roughly starlike with respect to the point $f(w)$.
\end{lemma}
\begin{proof}  Since $f:X\to Y$ is a rough $(\lambda,c)$-quasi-isometry, we see that for any $x'\in Y$ there is some point $x\in X$ such that
\be\label{l-2} d'(f(x),x')\leq c.\ee
By the assumption that  $X$ is $K$-roughly starlike with respect to $w\in X$, we see that there is a geodesic ray $\gamma$ emanating from $w$ to some $\xi\in \partial_\infty X$ and satisfying
\be\label{l-3}\dist(x,\gamma)\leq K.\ee
 Moreover, we see from \cite[Proposition 6.3]{BS} that there is an extension $f:\partial_\infty X\to \partial_\infty Y$ of $f$ with
$f(\xi)=\xi'\in \partial_\infty Y$. Then take another geodesic ray $\gamma'$ emanating from $w'=f(w)$ to $\xi'$. Because $f(\gamma)$ is a rough quasi-isometric ray, it follows from the extended stability theorem \cite[Theorem 6.32]{Vai-0}  that the Hausdorff distance
\be\label{l-4} \dist_\mathcal{H}(f(\gamma), \gamma')\leq C.\ee
Hence we obtain from (\ref{l-2}), (\ref{l-3}) and (\ref{l-4}) that
$$\dist(x',\gamma')\leq c+C+\lambda K+c=K',$$
as desired.
\end{proof}

We now pause to recall certain auxiliary definitions that we shall need. Suppose that $(X,d)$ is a $C$-annular quasiconvex, geodesic and proper metric space, and $\Omega\subsetneq X$ is a domain. For any $x\in\Omega$, denote $d_{\Omega}(x)=\dist(x,\partial \Omega)$. Let $0<\lambda\leq 1/2$. Following \cite[Chapter 7]{BHK} or \cite{HSX}, a point $x_0$ in $\Omega$ is said to be a {\it $\lambda$-annulus point} of $\Omega$, if there is a point $a\in\partial\Omega$ such that  $$t=d(x_0,a)=d_{\Omega}(x_0),$$
the annulus ${\mathbb{B}}(a,t/\lambda)\setminus \overline{{\mathbb{B}}}(a,\lambda t)$ is contained in $\Omega$.

If $x_0$ is not a $\lambda$-annulus point of $\Omega$, then  it is said to be a {\it $\lambda$-arc point} of $\Omega$. Then we have the following lemma.

\begin{lemma}\label{z-0}  Suppose that $(X,d)$ is a $C$-annular quasiconvex, geodesic and proper metric space, and $\Omega\subsetneq X$ is a bounded $\delta$-hyperbolic domain. Then $(\Omega,k)$ is $K$-roughly starlike with respect to some point $w\in \Omega$, where $K$ is a constant depending only on $C$ and $\delta$.
\end{lemma}
\begin{proof} Choose a point $w$ such that
$$d_{\Omega}(w)=\max\{d_{\Omega}(x):x\in \Omega\}.$$
We shall find a constant $K$ depending only on $C$ and $\delta$ such that for each $x\in \Omega$ there exists a quasihyperbolic geodesic ray $\alpha$ emanating from $w$ satisfying
$$\dist_k(x,\alpha)\leq K.$$

Let $\lambda=1/(2C)$.
Fix $x\in \Omega$, we consider two cases.

\begin{case} $x$ is a $\lambda$-arc point.\end{case}

 Then by \cite[Lemma 7.3]{HSX}, we see that there exist two points  $a,b\in\partial_\infty \Omega$ and a $C_1$-anchor $\gamma$ connecting $a, b$ with $x\in \gamma$ and $C_1=C_1(\lambda,C)$, where $\partial_\infty \Omega$ is the Gromov boundary of hyperbolic space $(\Omega,k)$. For the definition of anchor see \cite[Definition 7.2]{HSX}. Moreover, by the definition of anchor, we see that  $\gamma$ is a continuous quasihyperbolic $(C_1,C_1)$-quasigeodesic, that is,  $$\ell_k(\gamma[x,y])\leq C_1 k(x,y)+C_1$$ for all $x,y\in \gamma$.  On the other hand, we see from  \cite[Proposition 2.8]{BHK} that $(\Omega,k)$ is a proper geodesic metric space. Thus there is a quasihyperbolic geodesic line $\gamma_1$ connecting $a$ and $b$. It follows from \cite[Lemma 3.4]{HSX} that the quasihyperbolic Hausdorff distance
$$\dist_\mathcal{H}(\gamma, \gamma_1)\leq C_2,$$
where $C_2$ is a constant depending only on $C_1$ and $\delta$. This implies that there is a point $y\in\gamma_1$ with $$k(x,y)\leq C_2$$

Moreover, we may join $w$ to $a$ and $b$ by quasihyperbolic geodesic rays $\alpha_1$ and $\alpha_2$, respectively. Now by \cite[Lemma 3.1]{HSX}, we find that
$$\dist_k(y,\alpha_1\cup\alpha_2)\leq 24\delta,$$
and thus,
$$\dist_k(x,\alpha_1\cup\alpha_2)\leq 24\delta+C_2=K,$$
as desired.

\begin{case} $x$ is a $\lambda$-annular point. \end{case}

Thus there is a point $x_0\in\partial \Omega$ with
$$t=d(x_0,x)=d_{\Omega}(x),$$
the annulus ${\mathbb{B}}(x_0,t/\lambda)\setminus \overline{{\mathbb{B}}}(x_0,\lambda t)$ is contained in $\Omega$. Then choose a quasihyperbolic geodesic ray $\alpha$ emanating from $w$ to $x_0$. Since $d(w,x_0)\geq d_{\Omega}(w)\geq d_{\Omega}(x)=t$, there exist a point $z\in \alpha$ with $d(z,x_0)=t$. Because $X$ is $C$-annular quasiconvex, we see that there is a curve $\beta\subset {\mathbb{B}}(x_0,Ct)\setminus \overline{{\mathbb{B}}}(x_0,t/C)$ connecting $x$ and $z$ with
$$\ell(\beta)\leq Cd(x,z)\leq 2Ct.$$
Note that ${\mathbb{B}}(x_0,t/\lambda)\setminus \overline{{\mathbb{B}}}(x_0,\lambda t)$ is contained in $\Omega$ and $\lambda=1/(2C)$, it follows that $\beta\subset \Omega$ and for each $u\in \beta$, we have $d_{\Omega}(u)\geq t/(2C)$. Therefore, we have
$$k(x,z)\leq \ell_k(\beta)\leq 4C^2=K,$$
as required.
\end{proof}

\emph{Proof of Theorem \ref{thm-2}.} Firstly, it follows from \cite[Theorem 6.1]{BB03} that the Gehring-Hayman condition and the ball separation condition imply the Gromov hyperbolicty. It remains to show the necessity.  If $(\Omega,d)$ is bounded, then the assertion follows from \cite[Theorem 1.2]{KLM}. If $\Omega$ is unbounded, by \cite[Corollary 5.2]{KLM},  it suffices to check the rough starlikeness of $(\Omega, k)$.

Towards this end, take a point $a\in \partial \Omega$. Denote the sphericalization of metric space $(X,d)$ associated to the point $a$ by $(\dot{X},\widehat{d_a})$. Since $(X,d)$ is annular quasiconvex, it follows from \cite[Proposition 6.3 ]{BHX} that $(X,d)$ is quasiconvex. Hence, \cite[Lemma 3.4]{HRWZ} implies $(\Omega,d)$ is locally quasiconvex.  Since $a\in\partial \Omega$ and $(\Omega,d)$ is unbounded, it follows  from \cite[Theorem 4.12]{BHX} that the identity map $(\Omega,k)\to (\Omega,\widehat{k}_a)$ is bilipschitz, where $\widehat{k}_a$ is the quasihyperbolic metric of $\Omega$ with respect to the metric $\widehat{d_a}$. It thus implies that  $(\Omega,\widehat{k}_a)$ is Gromov hyperbolic, since $(\Omega,k)$ is Gromov hyperbolic  and the quasihyperbolic metric is a length metric. Hence, by Lemma \ref{lem-4}, we only need to show that $(\Omega,\widehat{k}_a)$ is roughly starlike.

According to \cite[Theorem 6.5(a)]{BHX}, it follows that the space $(\dot{X},\widehat{d_a})$ is quasiconvex and annularly quasiconvex.  Note that the rough starlikeness and Gromov hyperbolicity are preserved under bilipschitz mappings. So we may assume that $(\dot{X},\widehat{d_a})$ is geodesic, because the quasiconvexity condition implies that the identity map $(\dot{X},\widehat{d_a})\to (\dot{X},\ell_{\widehat{d}_a})$ is bilipschitz, where $\ell_{\widehat{d}_a}$ is the length metric of $(\dot{X},\widehat{d_a})$. Hence,  by Lemma \ref{z-0},
 we get $(\Omega,\widehat{k}_a)$ is roughly starlike, as desired.
\qed

\subsection{} The purpose of this subsection is to prove Theorem \ref{thm-3}. We also need some useful lemmas. The first one shows that bounded $\varphi$-uniformity condition implies the quasihyperbolic growth condition as follows.

\begin{lemma}\label{lem-5}
Let $(\Omega,d)$ be a locally compact, rectifiably connected and incomplete metric space. If $(\Omega,d)$ is bounded and $\varphi$-uniform, then there is an increasing function $\phi:[0,\infty)\to[0,\infty)$ such that for all $x\in \Omega$
$$k(w,x)\leq \phi\Big(\frac{d(w)}{d(x)}\Big),$$
where $w\in \Omega$ satisfies $d(w)=\max_{x\in \Omega}d(x)$ and $k$ is the quasihyperbolic metric of $\Omega$. In particularly, we can take $\phi(t)=\varphi(\frac{\diam \Omega}{d(w)}t)$. Here and hereafter, we use $d(x)$ to denote the distance from $x$ to the boundary of $\Omega$ with respect to the metric $d$.
\end{lemma}
%\begin{proof} Since $(\Omega,d)$ is bounded, there is a point $w\in \Omega$ that satisfies $d(w)=\max_{x\in \Omega}d(x)$.
%Because $(\Omega,d)$ is $\varphi$-uniform, by our choice of $w\in \Omega$, we have for all $x\in \Omega$
%$$k(w,x)\leq \varphi\Big(\frac{d(w,x)}{d(w)\wedge d(x)}\Big)=\varphi\Big(\frac{d(w,x)}{d(x)}\Big)\leq \phi\Big(\frac{d(w)}{d(x)}\Big),$$
%where $\phi(t)=\varphi(\frac{\diam \Omega}{d(w)}t)$, as required.
%\end{proof}

Secondly, we verify that the $\varphi$-uniformity condition is preserved under sphericalization.

\begin{lemma}\label{lem-6}
Let $(\Omega,d)$ be a locally compact, $c$-quasiconvex and incomplete metric space with $a\in\partial\Omega$. If $(\Omega,d)$ is unbounded and $\varphi$-uniform, then there exists a homeomorphism $\psi:[0,\infty)\to [0,\infty)$ such that the sphericalized space $(\Omega,\widehat{d}_a)$ associated to $a$ is $\psi$-uniform.
\end{lemma}
\begin{proof}  Since $(\Omega,d)$ is $c$-quasiconvex,  we  observe from \cite[Proposition 4.3]{BHX} that there are constants $\lambda\in(0,1/2)$ and $c_0\geq 1$ depending only on $c$ such that  the sphericalization  $(\Omega,\widehat{d}_a)$ is locally $(\lambda,c_0)$-quasiconvex. Moreover, we see from \cite[Theorem 4.12]{BHX} that the identity map $(\Omega,k)\to (\Omega,\widehat{k}_a)$ is $80c$-bilipschitz, where $\widehat{k}_a$ is the quasihyperbolic metric of $(\Omega,\widehat{d}_a)$.

To prove this lemma, we only need to find a homeomorphism $\psi:[0,\infty)\to [0,\infty)$ such that \be\label{l-new} \widehat{k}_a(x,y) \leq \psi\left(\frac{\widehat{d}_a(x,y)}{\widehat{d}_a(x)\wedge \widehat{d}_a(y)}\right)\ee for all $x,y\in\Omega$, where $\widehat{d}_a(x)$ denotes the distance from $x$ to the boundary of $\Omega$ with respect to the metric $\widehat{d}_a$. To this end,  we divide the proof into two cases.

\begin{case}$\widehat{d}_a(x,y)\leq \frac{\lambda}{3c_0}\widehat{d}_a(x)$.\end{case}

 A similar argument as in \cite[Lemma 3.8]{HRWZ} shows that
\be\label{z-2} \widehat{k}_a(x,y) \leq 3c_0\frac{\widehat{d}_a(x,y)}{\widehat{d}_a(x)}\leq 3c_0\frac{\widehat{d}_a(x,y)}{\widehat{d}_a(x)\wedge \widehat{d}_a(y)},\ee
as desired.

\begin{case}$\widehat{d}_a(x,y)> \frac{\lambda}{3c_0}\widehat{d}_a(x)$. \end{case}

Then we claim  that for all $x\in \Omega$,
\be\label{z-1} \widehat{d}_a(x)\leq \frac{2d(x)}{[1+d(x,a)]^2}.\ee

This can be seen as follows. Take a point $x_0\in\partial \Omega$ with $d(x)=d(x,x_0)$. We consider two possibilities. If $d(x_0,a)\leq \frac{1}{2}d(x,a)-\frac{1}{2}$, then we have
$$d(x)\geq d(x,a)-d(x_0,a)\geq \frac{1}{2}(1+d(x,a)),$$
and therefore,
$$\widehat{d}_a(x)\leq \widehat{d}_a(x,\infty)\leq \frac{1}{1+d(x,a)}\leq \frac{2d(x)}{[1+d(x,a)]^2}.$$
On the other hand, if $d(x_0,a)> \frac{1}{2}d(x,a)-\frac{1}{2}$, thus we have by (\ref{z-1.1}) that
\beq\nonumber
\widehat{d}_a(x)&\leq& \widehat{d}_a(x,x_0)
\\ \nonumber&\leq& \frac{d(x,x_0)}{[1+d(x,a)][1+d(x_0,a)]}
\\ \nonumber&\leq& \frac{2d(x)}{[1+d(x,a)]^2},
\eeq
as needed. This proves (\ref{z-1}).

Now by (\ref{z-1.1}) and (\ref{z-1}), we have for all $x,y\in \Omega$,
\beq\nonumber
\widehat{j}_a(x,y) &=& \log\Big(1+\frac{\widehat{d}_a(x,y)}{\widehat{d}_a(x)\wedge \widehat{d}_a(y)}\Big)
\\ \nonumber&\geq& \frac{1}{2}\log\Big(1+\frac{\widehat{d}_a(x,y)}{\widehat{d}_a(x)}\Big)\Big(1+\frac{\widehat{d}_a(x,y)}{\widehat{d}_a(y)}\Big)
\\ \nonumber&\geq& \frac{1}{2}\log \Big(1+\frac{d(x,y)(1+d(x,a))^2}{8d(x)}\Big)\Big(1+\frac{d(x,y)(1+d(y,a))^2}{8d(y)}\Big)
\\ \nonumber&>& \log\Big(1+\frac{d(x,y)}{8\sqrt{d(x)d(y)}}\Big)
\\ \nonumber&>& \log\frac{\sqrt{d(x)d(y)}+2d(x,y)}{\sqrt{d(x)d(y)}}-4\log2
\\ \nonumber&\geq& \frac{1}{2}\log\Big(1+\frac{d(x,y)}{d(x)}\Big)\Big(1+\frac{d(x,y)}{d(y)}\Big)-4\log 2
\\ \nonumber&\geq& \frac{1}{2}j(x,y)-4\log 2.
\eeq
Because the identity map $(\Omega,k)\to (\Omega,\widehat{k}_a)$ is $80c$-bilipschitz and $(\Omega,d)$ is $\varphi$-uniform, the above inequality implies that
\beq\label{z-3} \widehat{k}_a(x,y)&\leq& 80c k(x,y)
\\ \nonumber &\leq& 80c\varphi(e^{j(x,y)}-1)
\\ \nonumber &\leq&  80c\varphi\Big[256\Big(1+\frac{\widehat{d}_a(x,y)}{\widehat{d}_a(x)\wedge \widehat{d}_a(y)}\Big)^2-1\Big].
\eeq

Set $\psi(t)=3c_0t$ whenever $0\leq t\leq \lambda/3c_0$, and $\psi(t)=80cc_0\varphi(256(1+t)^2-1)$ whenever $t\geq \lambda/3c_0$. Therefore, by (\ref{z-2}) and (\ref{z-3}), we obtain \eqref{l-new}.
\end{proof}

\begin{remark} In \cite{LVZ19}, Li, Vuorinen and Zhou proved that quasim\"obius mappings preserve $\varphi$-uniform domains of $\mathbb{R}^n$. In the proof of Lemma \ref{lem-6}, we calculate the control function for our needs.
\end{remark}

\emph{Proof of Theorem \ref{thm-3}.} $(1)$ We assume first that $(\Omega,d)$ is bounded. Since $(\Omega,d,\mu)$ is Ahlfors $Q$-regular and $(\Omega,k)$ is roughly starlike and Gromov hyperbolic, we see from \cite[Theorem $5.1$]{KLM} that $(\Omega,d)$ satisfies the Gehring-Hayman condition. Because $(\Omega,d)$ is $\varphi$-uniform, it follows from Lemma \ref{lem-5} that $(\Omega,d)$ satisfies the quasihyperbolic growth condition with $\phi(t)=\varphi(\frac{\diam \Omega}{d(w)}t)$. That is,  for all $x\in \Omega$ we have
$$k(w,x)\leq \phi\Big(\frac{d(w)}{d(x)}\Big),$$
where $w\in \Omega$ satisfies $d(w)=\max_{x\in \Omega}d(x)$ and $k$ is the quasihyperbolic metric of $\Omega$. Note that the conditions
$$\int_1^\infty \frac{dt}{{\varphi^{-1}(t)}}<\infty\;\;\;\;\mbox{and}\;\;\;\;\int_1^\infty \frac{dt}{{\phi^{-1}(t)}}<\infty$$
are mutually equivalent.  Therefore, according to \cite[Theorem $1.1$]{La}, we immediately see that the Gromov boundary and the metric boundary of $\Omega$ are homeomorphic.

$(2)$ Assume that $(\Omega,d)$ is unbounded. Let $a\in \partial\Omega$. Denote by $(\Omega,\widehat{d}_a,\mu_a)$ the sphericalization of $(\Omega,d,\mu)$ associated to the point $a$. Let $\ell_{\widehat{d}_a}$ be the length metric of $\Omega$ with respect to the metric $\widehat{d}_a$. In order to show that there is a homeomorphism identification $\partial \Omega\cup \{\infty\}\to \partial_\infty \Omega,$ we need some preparations.

Firstly, we show that
\begin{claim}\label{z-4} the  identity map $(\partial \Omega\cup \{\infty\},d)\to (\partial \Omega\cup \{\infty\}, \ell_{\widehat{d}_a})$ is a homeomorphism.\end{claim}

Indeed, by the definition of $\ell_{\widehat{d}_a}$, we see that $d(x_n,a)\to \infty$ if and only if $\ell_{\widehat{d}_a}(x_n,a)\to 0$, as $n\to \infty.$ Thus it suffices to verify that for any sequence $\{x_n\}\subset \Omega$, $\{x_n\}$ is Cauchy in the metric $d$ if and only if $\{x_n\}$ is Cauchy in the metric $\ell_{\widehat{d}_a}$.

On one hand, by \cite[(2.11)]{LS} and the quasiconvexity of $(\Omega,d)$, it follows that
$$\ell_{\widehat{d}_a}(x_n,x_m)\leq \ell_d(x_n,x_m)\leq cd(x_n,x_m),$$
where $\ell_d$ is the length metric of $(\Omega,d)$.

On the other hand, we see from (\ref{z-1.1}) that
$$d(x_n,x_m) \leq 4[1+d(x_n,a)][1+d(x_m,a)] \ell_{\widehat{d}_a}(x_n,x_m),$$
as required. We have proved Claim \ref{z-4}.

Secondly, we check that
\begin{claim}\label{z-5}  $(\Omega,\widehat{k}_a)$ is  roughly starlike Gromov hyperbolic and $(\Omega,\ell_{\widehat{d}_a})$ satisfies the Gehring-Hayman condition, where $\widehat{k}_a$ is the quasihyperbolic metric of $\Omega$ with respect to the metric $\widehat{d}_a$.
\end{claim}

By \cite[Theorem $4.12$]{BHX}, we see that the identity map $(\Omega,k)\to (\Omega,\widehat{k}_a)$ is $80c$-bilipschitz. Since  $(\Omega, k)$ is Gromov hyperbolic, it follows from \cite[Page 402, Theorem 1.9]{BrHa}  that $(\Omega, \widehat{k}_a)$ is Gromov hyperbolic as well. Since $(\Omega,k)$ is roughly starlike, we see from  Lemma \ref{lem-4}  that $(\Omega,\widehat{k}_a)$ is also roughly starlike. Furthermore, according to \cite[Proposition 3.1]{LS}, we immediately find that the sphericalization $(\Omega,\widehat{d}_a,\mu_a)$ is Ahlfors $Q$-regular. Thus we obtain from \cite[Theorem 5.1]{KLM} that $(\Omega,\ell_{\widehat{d}_a})$ satisfies the Gehring-Hayman condition, which shows Claim \ref{z-5}.

Next, we claim that
\begin{claim}\label{z-6}
there is a natural homeomorphism identification
$$(\partial \Omega\cup \{\infty\}, \ell_{\widehat{d}_a})\to (\partial_\infty \Omega_a,\widehat{\rho}),$$
where $\partial_\infty \Omega_a$ is the boundary at infinity of Gromov hyperbolic space $(\Omega,\widehat{k}_a)$ and $\widehat{\rho}$ is an arbitrary Bourdon metric on $\partial_\infty \Omega_a$.
\end{claim}

Since $(\Omega,d)$ is $\varphi$-uniform, we see from Lemma \ref{lem-6} and its proof that both the spaces $(\Omega,\widehat{d}_a)$ and $(\Omega,\ell_{\widehat{d}_a})$ are $\psi$-uniform with $\psi(t)=80c\varphi(256(1+t)^2-1)$ for all $t\geq 1$. Thus it follows from Lemma \ref{lem-5} that there is an increasing function $\phi:[0,\infty)\to [0,\infty)$ and $w\in \Omega$ such that
$$\widehat{k}_a(w,x)\leq \phi\Big(\frac{\widehat{d}_a(w)}{\widehat{d}_a(x)}\Big)$$
with $\phi(t)=\psi(\frac{t}{\widehat{d}_a(w)})$, since $\diam_{\widehat{d}_a}(\Omega)\leq 1$ by  \eqref{z-1.1}. Moreover, it follows from the definitions of $\psi$ and $\phi$ that the conditions
$$\int_1^\infty \frac{dt}{\sqrt{\varphi^{-1}(t)}}<\infty\;\;\;\;\mbox{and}\;\;\;\;\int_1^\infty \frac{dt}{\phi^{-1}(t)}<\infty$$
are mutually equivalent. Now we observe that the conditions \cite[(1.3) and (1.5)]{La} with respect to the metric $\widehat{k}_a$ are verified. Therefore, it follows from Claim \ref{z-5} and \cite[Theorem 1.1]{La}  that there is a natural homeomorphism identification $(\partial \Omega\cup \{\infty\}, \ell_{\widehat{d}_a})\to (\partial_\infty \Omega_a,\widehat{\rho})$, which proves Claim \ref{z-6}.

Finally, since the identity map $(\Omega,k)\to (\Omega,\widehat{k}_a)$ is $80c$-bilipschitz, we see from \cite[Propositions 6.3 and 6.5]{BS} that there is a natural homeomorphism identification
$$(\partial_\infty \Omega,\rho) \to (\partial_\infty \Omega_a,\widehat{\rho}),$$
where $\partial_\infty \Omega$ is the Gromov boundary of hyperbolic space $(\Omega,k)$ and $\rho$ is a Bourdon metric defined on $\partial_\infty \Omega$.

Therefore, this together with Claims \ref{z-4} and \ref{z-6} shows that there is a natural homeomorphism identification $\partial \Omega\cup \{\infty\}\to \partial_\infty \Omega.$ Hence Theorem \ref{thm-3} is proved.
\qed

%%%%%%%%%%%%%%%%%%%


\begin{thebibliography}{99}

\bibitem{BB03} {\sc Z. M. Balogh and S. M. Buckley}, Geometric characterizations of Gromov hyperbolicity, \textit{Invent. Math.}, {\bf 153} (2003), 261--301.

\bibitem{BB}  {\sc Z. M. Balogh and S. M. Buckley,} Sphericalization and flattening,\textit{ Conf. Geom. Dyn.,} {\bf 9} (2005), 76--101.

\bibitem{BK02}  {\sc M. Bonk and B. Kleiner},
Rigidity for quasi-M\"{o}bius group actions, \textit{J. Differential Geom.,} {\bf 61} (2002),  81--106.

\bibitem{BHK}  {\sc M. Bonk, J. Heinonen and P. Koskela},
Uniformizing Gromov hyperbolic spaces, \textit{Ast\'{e}risque,} {\bf 270} (2001), viii+99 pp.

\bibitem{BS}  {\sc M. Bonk and O. Schramm}, Embeddings of Gromov hyperbolic spaces,
\textit{Geom. Funct. Anal.},  {\bf 10} (2000), 266--306.

\bibitem{Bou}  {\sc M. Bourdon},
Structure conforme au bord et flot g\'{e}od\'{e}sique d'un CAT(-1)-espace, \textit{Enseign. Math.,} {\bf 2} (1995),  63--102.

\bibitem{BrHa}  {\sc M. R. Bridson and A. Haefliger}, Metric spaces of non-positive curvature,
\textit{Grundlehren der Mathematischen Wissenschaften [Fundamental Principles of Mathematical Sciences],} (1999), 319, Springer-Verlag, Berlin.

\bibitem{BHX}  {\sc S. M. Buckley, D. Herron and X. Xie}, Metric space inversions, quasihyperbolic distance, and uniform spaces, \textit{Indiana Univ. Math. J.}, {\bf 57} (2008), 837--890.

\bibitem{BuSc}  {\sc S. Buyalo and V. Schroeder}, Elements of Asymptotic Geometry, \textit{EMS Monographs in Mathematics,} (2007).

\bibitem{GO}  {\sc F. W. Gehring and B. G. Osgood}, Uniform domains and the quasi-hyperbolic metric, \textit{J. Analyse Math.,} {\bf 36} (1979),50--74.

\bibitem{Gr87} {\sc M. Gromov}, Hyperbolic groups, \textit{Essays in group theory}, Math. Sci. Res. Inst. Publ., Springer, 1987,  75--263.

\bibitem{Ha}  {\sc U. Hamenst\"adt}, A new description of the Bowen-Margulis measure.
\textit{Ergodic Theory Dynam. Systems,} {\bf 9} (1989), 455--464.

\bibitem{Heer}  {\sc L. Heer}, Some Invariant Properties of Quasi-M\"obius Maps, \textit{Anal. Geom. Metr. Spaces,} {\bf 5} (2017), 69--77.

\bibitem{HSX}  {\sc D. Herron, N. Shanmugalingam and X. Xie}, Uniformity from Gromov hyperbolicity, \textit{Illinois J. Math.},  {\bf
52} (2008), 1065--1109.

\bibitem{Her06}  {\sc D. Herron}, Quasiconformal deformations and volume growth,
\textit{Proc. London Math. Soc.,} {\bf 92} (2006), 161--199.

\bibitem{HRWZ}  {\sc M. Huang, A. Rasila, X. Wang and Q. Zhou}, Semisolidity and locally weak quasisymmetry of homeomorphisms in metric spaces, \textit{ Stud. Math.},  {\bf 242} (2018), 267--301.

\bibitem{JJ}  {\sc J. Jordi,} Interplay between interior and boundary geometry in Gromov hyperbolic spaces, \textit{ Geom. Dedicata,} {\bf 149} (2010), 129--154.

\bibitem{KLM}  {\sc P. Koskela, P. Lammi and V. Manojlovi\'{c}}, Gromov hyperbolicity and quasihyperbolic geodesics,
\textit{Ann. Sci. \'{E}c. Norm. Sup\'{e}r},  {\bf 52} (2014), 975--990.

\bibitem{La}  {\sc P. Lammi}, Quasihyperbolic boundary condition: compactness of the inner boundary, \textit{ Illinois J. Math.}, {\bf 55} (2011), 1221--1233.

\bibitem{LS}  {\sc X. Li and N. Shanmugalingam}, Preservation of bounded geometry under sphericalization and fattening, \textit{ Indiana. Math J.}, (2015), 1303--1341.

\bibitem{LVZ19} {\sc Y. Li, M. Vuorinen and Q. Zhou}, Apollonian metric, uniformity and Gromov hyperbolicity, \textit{ Complex Var. Elliptic Equ.,}  {\bf 65} (2020), 215-228.


\bibitem{M}  {\sc I. Mineyev}, Metric conformal structures and hyperbolic dimension,
\textit{Conform. Geom. Dyn.}, {\bf 11} (2007), 137--163.

\bibitem{TV}  {\sc P. Tukia and J. V\"{a}is\"{a}l\"{a}}, Quasisymmetric embeddings of metric spaces,
\textit{Ann. Acad. Sci. Fenn. Ser. A I Math.,} {\bf 5} (1980),
97--114.

\bibitem{Vai-4} {\sc J. V\"{a}is\"{a}l\"{a}}, The free quasiworld,
freely quasiconformal and related maps in Banach spaces,
\textit{Quasiconformal geometry and dynamics $($Lublin 1996$)$,
Banach Center Publications}, (1999), 55--118.

\bibitem{Vai-5}  {\sc J. V\"{a}is\"{a}l\"{a}}, Quasi-M\"{o}bius maps, \textit{J. Anal. Math.,} {\bf 44} (1984/85), 218--234.

\bibitem{Vai-0}  {\sc J. V\"{a}is\"{a}l\"{a}}, Gromov hyperbolic spaces, \textit{Expo. Math.},  {\bf 23} (2005), 187--231.

\bibitem{Vu85}  {\sc M. Vuorinen}, Conformal invariants and quasiregular mappings, \textit{J. Anal. Math.},  {\bf 45} (1985), 69--115.

\bibitem{WZ}  {\sc X. Wang and Q. Zhou}, Quasim\"{o}bius maps, weakly quasim\"{o}bius maps and uniform perfectness in quasi-metric spaces, \textit{Ann. Acad. Sci. Fenn. Ser. AI Math.}, {\bf 42} (2017), 257--284.

\bibitem{WY} {\sc Y. Wang and J. Yang}, The pointwise convergence of $p$-adic M\"obius maps,
\textit{Sci. China Math.}, {\bf 57} (2014), 1--8.

\bibitem{W}  {\sc K. Wildrick}, Quasisymmetric parametrizations of two-dimensional metric planes, \textit{Proc. Lond. Math. Soc.,} {\bf 97} (2008), 783--812.

\bibitem{ZLL1}  {\sc Q. Zhou, Y. Li and X. Li}, Deformations on Symbolic Cantor Sets and Ultrametric Spaces, \textit{Bull. Malays. Math. Sci. Soc.,} (2019). https://doi.org/10.1007/s40840-019-00863-0

%\bibitem{ZLL2}  {\sc Q. Zhou, Y. Li and X. Li}, Sphericalization and flattening with their applications in quasimetric measure spaces, \textit{ Preprint}.

\end{thebibliography}
\end{document}